\documentclass{amsart}
\usepackage{amssymb}

\usepackage[british,UKenglish]{babel}
\usepackage{graphicx}         
\usepackage{fancyhdr}
\usepackage{rotating}
\usepackage{amsmath}
\usepackage{amssymb}
\usepackage{amsthm}
\usepackage{tikz}
\usepackage{url}
\usepackage{enumerate}
\usepackage{mathtools}
\usepackage{setspace}
\usepackage{color}
\usepackage[normalem]{ulem}
\usetikzlibrary{positioning}

\newtheorem*{notation*}{Notation}
  \newenvironment{proofclaim}{\begin{proof}[Proof of claim]}{\end{proof}}

\newtheorem{innercustomthm}{Theorem}
\newenvironment{customthm}[1]
  {\renewcommand\theinnercustomthm{#1}\innercustomthm}
  {\endinnercustomthm}
  
\usepackage{environ}
\newcounter{quote}

\NewEnviron{myquotenumber}{\vspace{3ex}\par
\refstepcounter{quote}%
\hfill\parbox{\dimexpr \textwidth-2cm}
{\centering\small\textit{\BODY}}
\hfill\llap{(\thequote)}\vspace{2ex}\par}

\def\Ind#1#2{#1\setbox0=\hbox{$#1x$}\kern\wd0\hbox to 0pt{\hss$#1\mid$\hss}
\lower.9\ht0\hbox to 0pt{\hss$#1\smile$\hss}\kern\wd0}

\def\notind#1#2{#1\setbox0=\hbox{$#1x$}\kern\wd0
\hbox to 0pt{\mathchardef\nn=12854\hss$#1\nn$\kern1.4\wd0\hss}
\hbox to 0pt{\hss$#1\mid$\hss}\lower.9\ht0 \hbox to 0pt{\hss$#1\smile$\hss}\kern\wd0}

\newtheorem{theorem}{Theorem}[section]
\newtheorem{lemma}[theorem]{Lemma}
\newtheorem{fact}[theorem]{Fact}
\newtheorem*{theorem*}{Theorem}
\newtheorem*{maintheorem*}{Main Theorem}
\newtheorem*{lemma*}{Lemma}
\newtheorem{corollary}[theorem]{Corollary}

\theoremstyle{definition}

\newtheorem{question}[theorem]{Question}
\newtheorem{definition}[theorem]{Definition}
\newtheorem{remark}[theorem]{Remark}
\newtheorem*{remark*}{Remark}

\newtheorem{claim}{Claim}
\newtheorem*{claim*}{Claim}

\title[Simple pseudofinite groups of dimension $3$]{Finite-dimensional pseudofinite groups of small dimension, without CFSG}

\author[U. Karhum\"{a}ki]{Ulla Karhum\"{a}ki}
\address{University of Helsinki}
\email{ukarhumaki@gmail.com}

\author[F. O. Wagner]{Frank O. Wagner}
\address{Universit\'{e} Claude Bernard Lyon 1; CNRS; Institut Camille Jordan}
\email{wagner@math.univ-lyon1.fr}

\thanks{The first author is funded by the Finnish Science Academy grant no.\ 338334.
The second author would like to acknowledge the support of the Science Committee of the Ministry of Science and Higher Education of the Republic of Kazakhstan (Grant No.\ AP19677451).}

\begin{document}

\begin{abstract}Any simple pseudofinite group $G$ is known to be isomorphic to a (twisted) Chevalley group over a pseudofinite field. This celebrated result mostly follows from the work of Wilson in 1995 and heavily relies on the classification of finite simple groups (CFSG). It easily follows that $G$ is finite-dimensional with additive and fine dimension and, in particular, that if ${\rm dim}(G)=3$ then $G$ is isomorphic to ${\rm PSL}_2(F)$ for some pseudofinite field $F$. We describe pseudofinite groups of fine and additive dimension  $\leqslant 3$ and, in particular, show that the classification $G \cong {\rm PSL}_2(F)$ is independent of CFSG.\end{abstract}

\maketitle

\section{Introduction}
\emph{Pseudofinite} groups (resp.\ fields) are infinite groups (fields) which satisfy every first-order property that is true in all finite groups (fields). In \cite{Ax1968}, Ax characterised pseudofinite fields in purely algebraic terms. Using Ax's result and the classification of finite simple groups (CFSG), Wilson proved that a simple pseudofinite group is elementarily equivalent to a (twisted) Chevalley group over a pseudofinite field  \cite{Wilson1995}. 

A structure is \emph{supersimple of finite SU-rank} if its definable sets (in some monster model) are equipped with a notion of dimension, called \emph{${\rm SU}$-rank}, taking integer values. It is well known \cite{Chatzidakis1992} that a pure pseudofinite field $F$ is supersimple of ${\rm SU}$-rank $1$. Hence, using Wilson's classification, the theory of any simple pseudofinite group $G$ is supersimple of finite ${\rm SU}$-rank \cite[Theorem 4.1]{Macpherson2018}. In particular, it is easy to see \cite[Proposition 6.1]{Elwesetal} that a simple pseudofinite group $G$ with ${\rm SU}(G)=3$ is isomorphic to ${\rm PSL}_2(F)$ for some pseudofinite field $F$. However, this observation heavily relies on CFSG. 

A more general notion was introduced by the second author in \cite[Definition 1.1]{Wagner2020}: A structure is \emph{finite-dimensional} if there is a dimension function ${\rm dim}$ from the collection of all interpretable sets (again, in some monster model) to $\mathbb{N} \cup \{-\infty\}$. In particular, supersimple structures of finite SU-rank are finite-dimensional, where the dimension is ${\rm SU}$-rank. In this case the dimension is \emph{additive} and \emph{fine} (see Section~\ref{sec:findim}). 

Our main result below describes pseudofinite groups of dimension $3$ (when the dimension is additive and fine) without invoking CFSG. As an immediate corollary (Corollary~\ref{corol:main}), we obtain that the identification above, $G \cong {\rm PSL}_2(F)$, can be proven without using CFSG. Thus we solve the problem proposed in \cite[p.\ 3, Question (3)]{Elwesetal} and in \cite[p.\ 171]{Macpherson2018}. 

\begin{theorem}[Without CFSG]\label{th:main}Let $G$ be a pseudofinite finite-dimensional group with additive and fine dimension. If ${\rm dim} (G) = 3$, then either $G$ is soluble-by-finite, or $\widetilde Z(G)$ is finite and $G/\widetilde Z(G)$ has a definable subgroup of finite index isomorphic to ${\rm PSL}_2(F)$ where $F$ is a  pseudofinite field. \end{theorem}


Let $G$ be a finite-dimensional pseudofinite group and assume that the dimension is additive and fine. To prove Theorem~\ref{th:main} we need to understand the structure of $G$ when its dimension is $1$ or $2$. It is easy to observe (see Section~\ref{sec:findim}) that $G$ satisfies the \emph{\emph{icc}$^0$-condition}: In $G$, any chain of intersections of uniformly definable subgroups, each having infinite index in its predecessor, has finite length,
bounded by some $n_\varphi$ depending only on the defining formula $\varphi$. In particular, $G$ satisfies the icc$^0$-condition for centralisers, called the \emph{$\widetilde{\mathfrak{M}}_c$-condition}, and therefore it follows from \cite[Corollaries 4.14 and 5.2]{Wagner2020} that if ${\rm dim}(G)=1$ then $G$ is finite-by-abelian-by-finite, and if ${\rm dim}(G)=2$ then $G$ is soluble-by-finite. However, only the result in the case ${\rm dim}(G)=1$ is proven independently of CFSG. Indeed, to the authors knowledge, there is no published result without CFSG which describes the structure of $G$ when ${\dim}(G)=2$. The second author has proven such result without CFSG and the pre-print where the proof appears is available at \cite{Wagner-preprint}. In this paper, we give a different proof. Namely, using similar arguments in dimension $2$ and $3$ (Section~\ref{section:identification}), we show that if $G$ is not soluble-by-finite, then in dimension $2$ we get a contradiction whereas in dimension $3$ we may identify ${\rm PSL}_2(F)$ as in Theorem~\ref{th:main}. Thus, in dimension $2$ we recover the following:

\begin{theorem}[Without CFSG, cf.\ \cite{Wagner-preprint}]\label{th:rank2} Let $G$ be a finite-dimensional pseudofinite group with additive and fine dimension. If ${\rm dim}(G) =2$, then $G$ is soluble-by-finite. Moreover, if $G$ is not finite-by-abelian-by-finite, then there is a definable normal subgroup $N$ of $G$ with ${\rm dim}(N) =1$. \end{theorem}

Our results suggest the following question.

\begin{question}Let $G$ be a non-abelian definably simple pseudofinite group. Can one show, independently of CFSG, that $G$ is finite-dimensional with additive and fine dimension?
\end{question}

We prove our results in Section~\ref{sec:proofs} and give all necessary background results in Section~\ref{sec:bground}. In Section~\ref{section:finite-group-theory} we discuss which results from finite group theory are needed in our proofs.

\section{Backgroud results}\label{sec:bground} Let $G$ be a group. We use the following standard terminology and notation: \begin{enumerate}[-]
\item If the largest soluble (resp.\ nilpotent) normal subgroup of $G$ exists, then it is denoted by ${\rm Rad}(G)$ (resp.\ ${\rm Fitt}(G)$) and called the \emph{soluble radical} (resp.\ \emph{Fitting subgroup}) of $G$. Note that if ${\rm Rad}(G)$ or ${\rm Fitt}(G)$ exists, then it is definable \cite[Theorem 1.1]{Houcine2013}.
\item The set of involutions of $G$ is denoted by $I(G)$.
\item If $P$ and $Q$ are properties of groups then $G$ is called \emph{$P$-by-$Q$} if $G$ has a normal subgroup $N$ so that $N$ has the property $P$ and $G/N$ has the property $Q$. Note that $P$-by-($Q$-by-$R$) implies ($P$-by-$Q$)-by-$R$, but may be strictly stronger. However, finite-by-abelian-by-finite is unambiguous: If $F$ is finite normal in $A$ which is normal of finite index in $G$ and $A/F$ is abelian, then $A'$ is finite characteristic in $A$, whence normal in $G$, so (finite-by-abelian)-by-finite implies finite-by-(abelian-by-finite). Moreover, $C_G(A')$ is a normal subgroup of finite index in $G$ and $C_A(A')$ is nilpotent of class at most $2$, so $G$ is nilpotent-by-finite. Similarly, (finite-by-nilpotent)-by-finite equals nilpotent-by-finite, and (finite-by-soluble)-by-finite equals soluble-by-finite (the nilpotency class or the derived length may go up by~1).
\end{enumerate}

\subsection{Needed results from finite group theory}\label{section:finite-group-theory}As one suspects, the Feit-Thompson Theorem is used in our results. 

\begin{theorem}[Feit-Thompson \cite{Feit-Thompson1963}]\label{th:Feit-Thompson}A finite group without involutions is soluble.\end{theorem}

The \emph{rank} of a finite group $G$ is the smallest cardinality of a generating set, and the \emph{$2$-rank} of $G$, denoted by  $m_2(G)$, is the largest rank of an elementary abelian $2$-subgroup. By the Feit-Thompson Theorem, if $m_2(G)=0$, then $G$ is solvable. If $m_2(G)=1$ them, by Burnside's transfer theorem and the Brauer-Suzuki Theorem \cite{Brauer-Suzuki1959}, the group $G$ is not simple.

Let then $G$ be a finite simple group of $m_2(G)=2$. Then the Sylow $2$-subgroups of $G$ are either dihedral groups, quasidihedral groups, wreathed groups, or isomorphic to a Sylow $2$-subgroup of the projective special unitary group ${\rm PSU}_3(4)$ \cite{Alperin-Brauer-Gorenstein1973}. Gorenstein and Walter proved that a simple group with dihedral Sylow $2$-subgroups is isomorphic either to ${\rm PSL}_2(q)$ for $q \geqslant 5$ or to the alternating group $A_7$ \cite{Gorenstein-Walter1995}. Alperin, Brauer and Gorenstein proved that a simple group with quasidiheral or wreathed Sylow $2$-subgroups is isomorphic either to ${\rm PSL}_3(q)$ or ${\rm PSU}_3(q)$ for odd $q$ or to the Mathieu group $M_{11}$ \cite{Alperin-Brauer-Gorenstein1970,Alperin-Brauer-Gorenstein1973}. Finally, Lyons showed that if a finite simple group $G$ has Sylow $2$-subgroups isomorphic to Sylow $2$-subgroups of ${\rm PSU}_3(4)$ then $G \cong {\rm PSU}_3(4)$ \cite{Lyons1972}. 

The results explained above are combined into a theorem in \cite{Gorenstein1983}:

\begin{theorem}[See e.g. {\cite[page 6]{Gorenstein1983}}]\label{th:finite2-rank2} Let $G$ be a non-abelian finite simple group with $m_2(G) \leqslant 2$. Then $m_2(G) =2$ and $G$ is isomorphic to one of the following groups $${\rm PSL}_2(q), {\rm PSL}_3(q), {\rm PSU}_3(q) \,\, {\rm for} \,\, {\rm odd} \,\, q, \,\, {\rm PSU}_3(4), A_7 \,\, {\rm or} \,\,  M_{11}.$$ \end{theorem}

Theorem~\ref{th:finite2-rank2} is used to recognise a simple pseudofinite group over a pseudofinite field $F$ of characteristic $\neq 2$. For the case of ${\rm char}(F)=2$ we use Bender's result below. A \emph{strongly embedded} subgroup of a finite group $G$ is a proper subgroup $E$ of even order such that $E \cap E^g$ has odd order whenever $g\in G \setminus E$. 

\begin{theorem}[Bender \cite{Bender1971}]\label{th:Bender-Suzuki}Let $G$ be a finite group containing a strongly embedded subgroup. Assume that $m_2(G) > 1$ and ${\rm Rad}(G)=1$. Then $G$ has a normal subgroup which is isomorphic to one of the following groups $${\rm PSL}_2(2^n), {\rm Sz}(2^{2n-1}), {\rm PSU}_3(2^n),\,\, {\rm for} \,\, n \geqslant 2.$$ \end{theorem}

\subsection{Commensurability and almost operators}\label{subsec:com} Let $G$ be a group, and $H, K \leqslant G $. The group $H$ is said to be \emph{almost contained} in $K$, written $H \lesssim K$, if $H \cap K$ has finite index in $H$. The subgroups $H,K \leqslant G$ are \emph{commensurable} if both $H \lesssim K$ and $K \lesssim H$. This is denoted by $H \simeq K$. If $H$ and $K$ are not commensurable then we write $H \not\simeq K$. A family $\mathcal{H}$ of subgroups of $G$ is \emph{uniformly commensurable} if there is $n\in \mathbb{N}$ so that $|H_1 : H_1\cap H_2| < n$ for all $H_1,H_2 \in \mathcal{H}$. Likewise, $K \leqslant G$ is uniformly commensurable to $\mathcal{H}$ if and only if $\mathcal{H}$ is uniformly commensurable and $K$ is commensurable to some (equiv.\ any) group in $\mathcal{H}$. The following result is due to Schlichting but the formulation we give here can be found in \cite[Theorem 4.2.4]{Wagner2000}.

\begin{theorem}[Schlichting's Theorem]\label{th:commensurable}Let $G$ be a group and $\mathcal{H}$ a uniformly commensurable family of subgroups of $G$. Then there is a subgroup $N$ of $G$ which is uniformly commensurable with all members of $\mathcal{H}$ and is invariant under all automorphisms of $G$ which fix $\mathcal{H}$ setwise. \end{theorem}

\begin{definition}Let $G$ be a group, and $H, K \leqslant G $.  We define the following subgroups:\begin{enumerate}
\item $\widetilde{N}_K(H)=\{k\in K : H\simeq  H^k\}$ is the \emph{almost normaliser} of $H$ in $K$. 
\item $\widetilde{C}_K(H)=\{k\in K : H  \lesssim  C_H(k) \}$ is the \emph{almost centraliser} of $H$ in $K$.
\item $\widetilde{Z}(G)= \widetilde{C}_G(G)$ is the \emph{almost centre} of $G$. 
\end{enumerate}The commensurability is \emph{uniform} in $\widetilde{N}_K(H)$ (resp.\ in $\widetilde{C}_K(H)$) if there is some $m\in \mathbb{N}$ so that if $H\simeq  H^k$ then $|H : H\cap H^k| < m$ (resp.\ if $H  \lesssim  C_H(k)$ then $|H : C_H(k)| < m$). \end{definition}

Notice that given a group $G$ and definable subgroups $H, K$ of $G $ so that the commensurability is uniform in $\widetilde{N}_K(H)$ (resp.\ in $\widetilde{C}_K(H)$), then $\widetilde{N}_K(H)$ (resp.\ $\widetilde{C}_K(H)$) is a definable subgroup of $G$.

If $\widetilde N_G(H)=G$ we say that $H$ is {\em almost normal} in $G$. 

\begin{fact}[Hempel {\cite[Theorem 2.17]{Hempel2020}}]\label{fact:com:sym}Let $G$ be a group and $H, K \leqslant G $ be definable subgroups so that the commensurability is uniform in $ \widetilde{C}_K(H)$ and in $\widetilde{C}_H(K)$. Then $H \lesssim  \widetilde{C}_H(K)$ if and only if $K \lesssim \widetilde{C}_K(H)$. \end{fact}

\subsection{Finite-dimensional groups}\label{sec:findim} A group $G$ is called \emph{finite-dimensional} if there is a dimension function ${\rm dim}$ from the collection of all interpretable sets in models of ${\rm Th}(G)$ to $\mathbb{N} \cup \{-\infty\}$ such that, for any formula $\phi(x,y)$ and interpretable sets $X$ and $Y$, the following hold: \begin{enumerate}
\item Invariance: If $a\equiv a'$ then ${\rm dim}(\phi(x,a))={\rm dim}(\phi(x,a'))$.
\item Algebraicity: If $X \neq \emptyset$ is finite then ${\rm dim}(X)=0$, and ${\rm dim}(\emptyset)=-\infty$.
\item Union: ${\rm dim}(X \cup Y)={\rm max}\{{\rm dim}(X), {\rm dim}(Y)\}$.
\item Fibration: If $f: X \rightarrow Y$ is an interpretable map such that ${\rm dim}(f^{-1}(y)) \geqslant d$ for all $y\in Y$ then ${\rm dim}(X)\geqslant {\rm dim}(Y)+d$.
\end{enumerate} The dimension of a tuple $a$ of elements over a set $B$ is defined as
$${\rm dim}(a/B) := {\rm inf}\{ {\rm dim}(\phi(x)) : \phi \in {\rm tp}(a/B) \}.$$ We say that the dimension is \begin{itemize}
\item \emph{additive} if ${\rm dim}(a,b/C) = {\rm dim}(a/b,C) + {\rm dim}(b/C)$ holds for any
tuples $a$ and $b$ and for any set $C$; and
\item \emph{fine} if ${\rm dim}(X)=0$ implies that $X$ is finite. 
\end{itemize}

In this paper we work with a finite-dimensional group with additive and fine dimension. Below we list those properties of such groups which will be used repeatedly throughout the paper. For any further details on (finite-)dimensional groups we refer to \cite{Wagner2020}. 

\begin{fact}[Lascar equality]\label{fact:Lascar-eq}Let $G$ be a finite-dimensional group with additivite and fine dimension and $H \leqslant G$ be a definable subgroup. Then ${\rm dim}(G)={\rm dim}(H)+{\rm dim}(G/H)$.\end{fact}
\begin{proof}The map $G \rightarrow G/H$ has
fibres of dimension ${\rm dim}(H)$. Hence, by the additivity, we get ${\rm dim}(G)={\rm dim}(H)+{\rm dim}(G/H)$ \cite[Remark 1.4]{Wagner2020}. \end{proof}

\begin{lemma} Let $G$ be a finite-dimensional group with additive and fine dimension. Then $G$ satisfies the chain condition on interscetions of uniformly definable subgroups, icc$^0$:
\begin{enumerate}[icc$^0$:]
\item Given a family $\mathcal{H}$ of uniformly definable subgroups of $G$, there is $m < \omega$ so
that there is no sequence $\{H_i : i \leqslant m\} \subset \mathcal{H}$ with $|\bigcap_{i < j}H_i : \bigcap_{i \leqslant j}H_i |\geqslant m$ for all $j \leqslant m$.
\end{enumerate} 
(If $\mathcal H$ is the family of centralisers of elements, this is called the 
\emph{$\widetilde{\mathfrak{M}}_c$-condition}.)\end{lemma}
\begin{proof}
Assume to the contrary that $G$ does not satisfy the icc$^0$. Then, by compactness, the condition $$|\bigcap_{i < m}H_i : \bigcap_{i \leqslant m}H_i |\geqslant m\text{ for all }m< \omega$$ is a consistent first-order condition on the parameters needed to define the groups $(H_i : i < \omega)$. So there is a family $\{H_i : i \le {\rm dim}(G)\} \subset \mathcal{H}$ so that $$G>H_0 > H_1\cap H_0 > \ldots > \bigcap_{i \le{\rm dim}(G)} H_i$$ is a descending chain of length ${\rm dim}(G)+2$ of definable subgroups of $G$, each having infinite index in its predecessor, contradicting Fact~\ref{fact:Lascar-eq}.\end{proof}

\begin{corollary} Let $G$ be an icc$^0$group. If $H \leqslant G$ is definable, then $\widetilde{N}_G(H)$ is also definable.\end{corollary}
\begin{proof} We need to show that the commensurability is uniform for conjugates of $H$. So suppose not. Let $m$ be given by the the icc$^0$-condition, and choose a maximal
chain $$H > H\cap H^{g_0} > \ldots > H \cap \bigcap_{i=0}^k H^{g_i}$$ with every group of finite index at least $m$ in its predecessor. Then $k < m$. Now, if there were $g\in G$ so that $\infty > |H:H\cap H^g| > m\,|H: H \cap \bigcap_{i=0}^{k} H^{g_i}|$, then $$|H \cap \bigcap_{i=0}^{k} H^{g_i} :H \cap \bigcap_{i=0}^{k} H^{g_i}\cap H^g| > m,$$
contradicting the maximality of $k$. Since $|H^g:H^g\cap H|=|H:H\cap H^{g^{-1}}|$, this finishes the proof.\end{proof}

\begin{fact}[{\cite[Lemma 4.4]{Wagner2020}}] Let $G$ be an $\widetilde{\mathfrak M}_c$-group. If $H \leqslant G$ is definable, then $ \widetilde{C}_G(H)$ is also definable.\end{fact}

Recall that if in a group $G$ there is a finite bound on the size of its conjugacy classes, then the derived subgroup $G'$ is finite. In particular, in any $\widetilde {\mathfrak M}_c$-group the almost centre $\widetilde Z(G)$ is finite-by-abelian.

\begin{fact}[Hempel {\cite[Theorem 4.7.]{Hempel2020}}]\label{th:Fitting}Let $G$ be an $\widetilde{\mathfrak{M}}_c$-group. Then the Fitting subgroup ${\rm Fitt}(G)$ exists and is therefore definable.\end{fact}

For an arbitrary group $H$ denote by $R(H)$ the group generated by all soluble normal subgroups of $H$; note that if $R(H)$ is soluble then $R(G)={\rm Rad}(H)$. To prove our results we need the following fact.

\begin{theorem}\label{fact:solv-rad}Let $G$ be a finite-dimensional group with additive and fine dimension. Then the soluble radical ${\rm Rad}(G)$ exists and is therefore definable. 
\end{theorem}
\begin{proof}
We prove the claim by induction on the dimension. If ${\rm dim}(G)=0$ then $G$ is finite and $R(G)={\rm Rad}(G)$. Assume that the claim holds when ${\rm dim}(G)=n-1$ and let ${\rm dim}(G)=n$. Now, by Fact~\ref{th:Fitting}, the Fitting subgroup ${\rm Fitt}(G)$ exists. Note that $R({\rm Fitt}(G))={\rm Rad}({\rm Fitt}(G))={\rm Fitt}(G)$. If ${\rm Fitt}(G)$ is infinite then ${\rm dim}(G/{\rm Fitt}(G))< n$ and $R(G/{\rm Fitt}(G))={\rm Rad}(G/{\rm Fitt}(G))$ by the inductive assumption. Since ${\rm Fitt}(G)$ is a normal and soluble subgroup of $G$, we get that $R(G)={\rm Rad}(G)$.

Now, we may assume that ${\rm Fitt}(G)$ is finite.  Set $G_0:=C_G({\rm Fitt}(G))$ and $Z:=Z({\rm Fitt}(G))$. Then $G_0$ is a normal finite index subgroup of $G$. If $AZ/Z \leqslant G_0/Z$ is an abelian normal subgroup then $AZ\leqslant {\rm Fitt}(G_0)\le {\rm Fitt}(G)\cap G_0=Z$. So $G_0/Z$ has no non-trivial abelian normal subgroups and hence $R(G_0)={\rm Rad}(G_0)=Z$. Now $R(G)\cap G_0=R(G_0)$ and since $|R(G): R(G)\cap G_0| \leqslant |G:G_0|$ is finite, $R(G)$ is finite, whence soluble. \end{proof}

\subsection{Pseudofinite groups}\label{Sec:pf} We denote by $\mathcal{L}_{gr}$ the language of groups.

\begin{definition}A \emph{pseudofinite} group is an infinite group which satisfies every first-order sentence of $\mathcal{L}_{gr}$ that is true of all finite groups. \end{definition}

Note that by \L o\'s' Theorem, an infinite group (resp.\ $\mathcal{L}$-structure) is pseudofinite if and only if it is elementarily equivalent to a non-principal ultraproduct of finite groups ($\mathcal{L}$-structures) of increasing orders, see \cite{Macpherson2018}.

Typical examples of pseudofinite groups are torsion-free divisible abelian groups, infinite extraspecial groups of exponent $p >2$ and rank $n$ and (twisted) Chevalley groups over pseudofinite fields (recall that \emph{using} CFSG, Wilson \cite{Wilson1995} proved that a simple pseudofinite group is elementarily equivalent to a (twisted) Chevalley group $X(F)$ over a pseudofinite field $F$. Further, as explained in \cite{Macpherson2018}, results by Point \cite{Point1999} allow one to conclude that a simple group is pseudofinite if and only if it is elementarily equivalent to $X(F)$; Ryten \cite[Chapter 5]{Ryten2007} has generalised this by showing that `elementarily equivalent' can be strengthened to `isomorphic'. Note that replacing $\equiv$ by $\cong$ is done without further use of CFSG). 

We will use the following theorem repeatedly, often without referring to it.

\begin{theorem}[Wagner {\cite[Corollary 4.14]{Wagner2020}}]\label{th:infinite-abelian}Let $G$ be a finite-dimensional pseudo\-finite group with additive and fine dimension. If ${\rm dim}(G)=1$, then $G$ is finite-by-abelian-by-finite.
\end{theorem}

\begin{remark}\label{remark:FItting}Let $G$ be a pseudofinite finite-dimensional group with additive and fine dimension. If ${\rm dim}(G)=1$, then, by the above, $\widetilde{Z}(G)$ is of finite index in $G$ and the commutator group $\widetilde{Z}(G)'$ is finite. So $G$ contains a definable infinite (finite and characteristic)-by-abelian characteristic subgroup.\end{remark}

One of the earliest (1955) results on centralisers of involutions in finite groups states that, up to isomorphism, there are only a finite number of finite simple groups with a given centraliser of an involution. This was shown by Brauer and Fowler \cite{Brauer-Fowler}. The following result (which, like the result by Brauer and Fowler, is independent of CFSG) provides an alternative proof for this fact (see \cite[Corollary 2.5]{Hempel-Palacin2020}):

\begin{fact}[Hempel-Palac\'{i}n {\cite[Lemma 2.3]{Hempel-Palacin2020}}]\label{hempel-palacin}Let $G$ be a pseudofinite group and assume that $\widetilde Z(G)=1$. Then the centraliser $C_G(i)$ is infinite for any $i\in I(G)$.
\end{fact}

\section{Proofs of our results}\label{sec:proofs} Before starting our proofs, we still need some definitions and observations.

A group $G$ is called \emph{definably simple} if it has no proper non-trivial definable normal subgroup. It is \emph{semisimple} if it has no non-trivial abelian normal subgroup. Note that if $A$ is a non-trivial normal abelian subgroup and $1\not=a\in A$, then $Z(C_G(a^G))$ is a definable non-trivial normal abelian subgroup, so semisimplicity is the same as definable semisimplicity. Clearly, a semisimple group has trivial soluble radical. Moreover, if $G$ is semisimple and $N$ is a normal subgroup of finite index, then $N$ is also semisimple: If $A$ were a normal abelian subgroup of $N$, the finitely many $G$-conjugates of $A$ would generate a nilpotent subgroup normal in $G$, whose centre would would be a non-trivial abelian normal subgroup of $G$.

The \emph{socle} ${\rm Soc}(G)$ of a finite group $G$ is the subgroup generated by all minimal normal non-trivial subgroups.

A definable subgrop $H$ of $G$ is \emph{strongly embedded} if $H$ has involutions, but $H\cap H^g$ does not for any $g\in G\setminus H$. By \L o\'s' Theorem, if $G=\prod_{i\in I}G_i/ \mathcal{U}$ is pseudofinite, then $H$ is strongly embedded if and only if $H_i$ is strongly embedded in $G_i$ for almost all $i$.

\subsection{Useful lemmas}
We shall call an infinite group \emph{almost simple} if it is not abelian-by-finite and has no definable normal subgroup of infinite index.

\begin{lemma}\label{lemma:almost} An almost simple icc$^0$-group $G$ is semisimple. Moreover, $\widetilde Z(G)$ is trivial, and if $H$ is an infinite definable subgroup of infinite index, then $\widetilde N_G(H)<G$. If $H$ is a definable infinite soluble-by-finite subgroup, then $\widetilde N_G(H)$ has infinite index in~$G$.\end{lemma}
\begin{proof} If $A$ is an abelian normal subgroup of $G$ and $1\not=a\in A$, then $Z(C_G(a^G))$ is a definable abelian normal subgroup. It has either finite index or smaller dimension than $G$, contradicting almost simplicity in both cases.

Now $\widetilde Z(G)'$ is finite normal in $G$, whence trivial, so $\widetilde Z(G)$ is abelian normal, and must be trivial as well.

If an infinite definable infinite index subgroup $H$ is almost normal in $G$, then by Theorem \ref{th:commensurable} it is commensurable with a definable normal subgroup $N$ of $G$, again contradicting almost simplicity as commensurability implies that $N$ is also of infinite index in $G$.

Finally suppose $H$ is definable infinite and soluble-by-finite. If $\widetilde N_G(H)$ is of finite index in $G$ then $$N=\bigcap_{g\in G}\widetilde N_G(H)^g$$
is a normal subgroup of finite index in $G$. By Theorem \ref{th:commensurable} there is $\bar H$ normal in $\widetilde N_G(H)$ and commensurable with the soluble-by-finite group $H$. Now $\bar H$ is again soluble-by finite, as is $\bar H\cap N$. The product of the finitely many $G$-conjugates of $\bar H\cap N$ is a definable subgroup $S$ which is again soluble-by-finite, and normal in $G$. Hence ${\rm Rad}(S)$ is a non-trivial definable normal soluble subgroup of $G$, contradicting semisimplicity.\end{proof}
We shall in particular apply Lemma~\ref{lemma:almost} when $G$ is finite-dimensional and ${\rm dim}(H)=1$.

\begin{lemma}\label{lemma:involutions}Let $G \equiv \prod_{i\in I}G_i/\mathcal{U}$ be a non-abelian semisimple pseudofinite finite-dimensional group with additive and fine dimension. Then $G$ has involutions. Assume further that ${\rm dim}(G) \leqslant 3$ and that at least one of the following holds. \begin{enumerate}
\item $ m_2(G) \leqslant 2$.
\item $G$ contains a definable strongly embedded subgroup.
\end{enumerate} Then ${\rm dim}(G) = 3$ and $G$ has a definable normal subgroup of finite index isomorphic to ${\rm PSL}(F)$ where $F$ is a pseudofinite field.\end{lemma}

\begin{proof}Assume first that $G=\prod_{i\in I}G_i/\mathcal{U}$. We shall argue modulo $\mathcal{U}$ even if only implicitly. 

By semisimplicity $C_{G_i}({\rm Soc}(G_i))=1$; thus $G_i \hookrightarrow {\rm Aut}({\rm Soc}(G_i))$. Now the non-trivial socle ${\rm Soc}(G_i)$ of the finite group $G_i$ is a direct product of non-abelian simple finite groups; the Feit-Thompson Theorem then implies that ${\rm Soc}(G_i)$ has involutions for almost all $i$, and so does $G$.

Assume now that ${\rm dim}(G) \leqslant 3$. If (1) holds, then, by Theorem~\ref{th:finite2-rank2}, each simple factor in the direct product ${\rm Soc}(G_i)$ is of $2$-rank $2$ and hence ${\rm Soc}(G_i)$ is simple. Further, again by Theorem~\ref{th:finite2-rank2}, ${\rm Soc}(G_i)=X(q_i)$ where $X\in \{{\rm PSL}_2, {\rm PSL}_3, {\rm PSU}_3\}$ and $q_i$ is odd. If (1) does not hold and (2) holds, then Theorem~\ref{th:Bender-Suzuki} implies that $G_i$ has a normal subgroup $X(q_i)$ where $X\in \{{\rm PSL}_2, {\rm PSU}_3, {\rm Sz}\}$ and $q_i$ is a power of $2$. Note that since $|G_i|\le|{\rm Aut}(X(q_i))|$, the $q_i$ grow without a bound when $i$ varies. Therefore, in either of the two cases, by \cite{Ellers-Gordeev1998}, there is $x_i \in X(q_i)$ so that $X(q_i) =  x_i^{X(q_i) } x_i^{X(q_i) }$. (The result in \cite{Ellers-Gordeev1998} states that Thompson's Conjecture holds for finite simple (twisted) Chevalley groups $X(q)$, provided that $q > 8$. This does not use CFSG. Note however that while Thompson's Conjecture is known to hold for `almost all' finite simple groups (see \cite[Introduction]{Ellers-Gordeev1998}), this more general result uses a case-by-case analysis provided by CFSG.) Since $X(q_i) \unlhd G_i$ we see that $X(q_i) = x_i^{G_i} x_i^{G_i}$, and there is a definable normal subgroup $N=\prod_{i\in I}X(q_i)/\mathcal{U}=X(F)$, where $F=\prod_{i\in I} \mathbb{F}_{q_i}/\mathcal{U}$ is a pseudofinite field and $X\in \{{\rm PSL}_2,{\rm PSL}_3,{\rm PSU}_3,{\rm Sz}\}$. As ${\rm dim}(N) \leqslant 3$, we have ${\rm dim}(N) = 3$ and $X= {\rm PSL}_2$ by \cite[Proposition 6.1]{Elwesetal}.

Now the field $F$ and the linear structure $X(F)$ are definable in $G$. It follows that the result also holds for any group elementarily equivalent to $G$ (see \cite[Chapter~5]{Ryten2007}).\end{proof}

\begin{lemma}\label{lemma:nilp-by-finite}Let $G$ be a pseudofinite finite-dimensional group with additive and fine dimension. Assume that ${\rm dim} (G) \leqslant 3$, and that there is a definable subgroup $H < G$ with ${\rm dim}(H)=2$ which is not almost normalised by $G$. Then either $\widetilde{Z}(H)$ is finite, or ${\rm dim}(\widetilde Z(G))\ge 1$.\end{lemma}
\begin{proof} Consider $x\in G \setminus\widetilde N_G(H)$. Then ${\rm dim}(H\cap H^x)\ge 1$, as otherwise ${\rm dim}(HH^x)=4 > {\rm dim} (G)$. Also ${\rm dim}(H\cap H^x) < 2$ by the choice of $x$. So ${\rm dim}(H\cap H^x)=1$ and $H\cap H^x$ is finite-by-abelian-by-finite. Thus $\widetilde{Z}(H\cap H^x)\simeq H\cap H^x$.

Now assume that $\widetilde{Z}(H)$ is infinite, and put $L=\widetilde{Z}(H\cap H^x) \cap \widetilde{C}_H(\widetilde{Z}(H))$. Since $H\lesssim \widetilde{C}_H(\widetilde{Z}(H))$ (Fact~\ref{fact:com:sym}), we have 
$L\simeq \widetilde{Z}(H\cap H^x)$.

Suppose first that $\widetilde{Z}(H) \cap \widetilde{Z}(H\cap H^x)$ is finite. If $y\in L$, then $C_H(y)$ almost contains both $\widetilde{Z}(H)$ and $\widetilde{Z}(H\cap H^x)$, so ${\rm dim}(C_H(y))=2$ and $y\in \widetilde{Z}(H)$. Thus $\widetilde{Z}(H\cap H^x) \lesssim L\lesssim \widetilde{Z}(H)$.

On the other hand, if $\widetilde{Z}(H) \cap \widetilde{Z}(H\cap H^x)$ is infinite, we again get $\widetilde{Z}(H\cap H^x)\lesssim \widetilde{Z}(H)$ since ${\rm dim}(\widetilde Z(H\cap H^x))=1$.

Similarly $\widetilde{Z}(H\cap H^x) \lesssim \widetilde{Z}(H^x) $. So 
$\widetilde{Z}(H\cap H^x) \lesssim \widetilde{Z}(H)\cap\widetilde Z(H^x)\le\widetilde Z(G)$, and ${\rm dim}(\widetilde Z(G))\ge 1$.\end{proof}

\subsection{The identification lemma}\label{section:identification}In this section we prove Lemma~\ref{lemma:identification} which plays a key role in the proof of both Theorem~\ref{th:rank2} and Theorem~\ref{th:main}.
We use similar arguments in dimension $2$ and $3$, but in dimension $2$ this will lead to a contradiction, while in dimension $3$ we shall identify ${\rm PSL}_2(F)$. However, the case of dimension $3$ is more complicated due to the possible existence of definable proper subgroups of dimension $2$. So we start with the analysis of such subgroups.

\begin{lemma}\label{lemma:dim2subgroups}Let $G$ be an almost simple pseudofinite finite-dimensional group with additive and fine dimension. Assume that ${\rm dim} (G)=3$ and that all definable subgroups of dimension $2$ are soluble-by-finite. If $L$ is a definable subgroup of dimension $2$, then:\begin{enumerate}
\item\label{lemma:item1} $B=\widetilde N_G(L)$ is a maximal definable subgroup of dimension $2$. Moreover, $B=N_G(B)=\widetilde N_G(B)$; if $Z=\widetilde Z(B)$ then $Z$ is finite, and $B/Z$ is a Frobenius group with Frobenius kernel $U/Z$ and Frobenius complement $T/Z$, where $U=\widetilde C_G(U)$ and $T=\widetilde C_B(T)$ have dimension $1$ and are finite-by-abelian. 
\item\label{lemma:item2} There is a pseudofinite field $F$ such that $U/Z\cong F^+$ and $T/Z$ embeds into $T^\times$ as a subgroup of finite index. In particular $T/Z$ is abelian, and $T$ has only finitely many elements of any given order. 
\item\label{lemma:item3} If $C=U_CT_C$ is another (maximal) subgroup of dimension $2$ and $Z_C:=\widetilde{Z}(C)\not=1$, then $U^g\not\simeq T_C$ for any $g\in G$.
\item\label{lemma:item4} If $g\in G\setminus B$ then $U\cap U^g=1$. If $B$ contains involutions, then either $B$ is strongly embedded, or there is an involution in $T\setminus Z$ and no involution in $U\setminus Z$. 
\item\label{lemma:item5} Suppose $Z\not=1$. If $g\in G\setminus B$ then $B\cap B^g$ is a finite index subgroup of $T^u$ for some $u\in U$. Moreover, if $x\in B$ with ${\rm dim}(C_G(x))=2$ then $x\in Z$.\end{enumerate}\end{lemma}

\begin{proof} Put $N=\widetilde N_G(L)$. Then ${\rm dim}(N)=2$ by Lemma \ref{lemma:almost}. If $K$ is any group commensurable with $L$ then $K\le \widetilde N_G(K)=\widetilde N_G(L)$. This shows maximality of $N$, as well as $N=\widetilde N_G(N)=N_G(N)$.

Since $\widetilde Z(G)$ is trivial and $N$ is not almost normal in $G$, Lemma \ref{lemma:nilp-by-finite} yields that $Z:=\widetilde Z(N)$ is finite. But $\widetilde Z(N)$ contains all finite normal subgroups of $N$. As $N$ is soluble-by-finite by assumption, there is $a\in N\setminus\widetilde Z(N)$ such that $\langle a^N\rangle/\widetilde Z(N)$ is abelian; it is infinite since $a\notin\widetilde Z(N)$. Then $C_N(a^N/\widetilde Z(N))$ is an infinite definable subgroup containing $a^N$; if it were of dimension $2$ then $a^N\subseteq Z$, a contradiction. Thus ${\rm dim}(C_N(a^N/\widetilde Z(N)))=1$. Define: \begin{enumerate}[(1)]
\item $U:=\{x\in N : C_N(x)\gtrsim C_N(a^N/\widetilde Z(N))\}=\widetilde C_N(C_N(a^N/\widetilde Z(N)))$.
\end{enumerate}
Clearly $U$ only depends on the commensurability class of $C_N(a^N/\widetilde Z(N))$. If ${\rm dim}(U)=2$, then $N\lesssim U=\widetilde C_N(C_N(a^N/\widetilde Z(N)))$, so $C_N(a^N/\widetilde Z(N))\lesssim\widetilde C_N(N)=Z$ which is finite, a contradiction. Hence $U\simeq C_N(a^N/\widetilde Z(N))$, and $U=\widetilde Z(U)$ is finite-by-abelian. Clearly $N\le N_G(U)\le\widetilde N_G(U)$, and we have equality by maximality and almost simplicity of $G$.

Now $N/U$ is $1$-dimensional, whence finite-by-abelian-by-finite. For any $x_0\in\widetilde C_N(N/U)\setminus U$ the orbit $(x_0U)^N$ in $N/U$ is finite. Then ${\rm dim}(x_0^N)={\rm dim}(U)=1$ and $C_N(x_0)$ has dimension $1$. Note that $x_0\notin U$ implies finiteness of $C_N(x_0)\cap U$. So we can define the following:
\begin{enumerate}[(2)]
\item $T:= \{x\in N : C_N(x) \gtrsim C_N(x_0)\}=\widetilde C_N(C_N(x_0))$. \end{enumerate}
\begin{enumerate}[(3)]
\item $B:=UT$.
\end{enumerate}
Note that $\widetilde N_G(B)=\widetilde N_G(N)=N$, as $B$ has finite index in $N$.
As ${\rm dim}(C_N(x_0)U)=2$, the index of $C_N(x_0)U$ in $N$ is finite, so 
$$Z=\widetilde Z(N)=\widetilde C_N(C_N(x_0)U)=\widetilde C_N(C_N(x_0))\cap\widetilde C_N(C_N(a^N/\widetilde Z(N)))=T\cap U.$$
Thus ${\rm dim}(T)=1$; it follows that $T\simeq C_N(x_0)$ and $T=\widetilde C_N(T)$ is finite-by-abelian. Moreover, any $N$-conjugate of $T$ which is commensurable with $T$ must be equal to $T$, and $\widetilde N_N(T)=N_N(T)$.
For any $g\in N\setminus N_N(T)$ we have $T\cap T^g=Z$, since $C_N(h)$ has dimension $2$ for any $h\in T\cap T^g$. Moreover $N_B(T)=N_U(T)\times_Z T$; it cannot have dimension $2$ as otherwise it would be finite-by-abelian-by-finite, contradicting finiteness of $\widetilde Z(N)$. Hence $N_U(T)$ is finite and $N_U(T)\le\widetilde C_U(T)=Z$. Thus $N_B(T)=T$.

Put $A:=U/Z$, $H:=T/Z$ and $\hat{B}:=B/Z$. Clearly $A\cap H=1$ and $\widetilde{Z}(\hat{B})=1$, so $\hat{B}=A \rtimes H$ and $H \cap H^g=1$ for any $g\in \hat B\setminus H$. This means that $\hat{B}$ is a pseudofinite Frobenius group with a definable Frobenius kernel $A$. Therefore, by the structure of finite Frobenius groups, the conjugates of $H$ cover $\hat B\setminus A$, and for any definable subgroup $K$ of $\hat{B}$ with 
$K\cap A=1$ and $AK=\hat{B}$ there is $a\in A$ with $K=H^a$. 

Now, let $M\le N$ be definable of dimension $1$ with $M\cap U$ finite. As above $T_1:=\widetilde C_N(M)$ is of dimension $1$, commensurable with $M$, satisfies $T_1\cap U=Z$, and $N_{UT_1}(T_1)=T_1$. Then $\hat B_1=UT_1/Z\le\hat B$ is also a Frobenius group with Frobenius kernel $A$. Since $H_1=H\cap\hat B_1$ satisfies 
$H_1\cap A=1$ and $AH_1=\hat B_1$, there is $u\in U$ with $H^u_1=T_1/Z$.
Thus $T^u\simeq (T\cap B_1)^u=T_1\simeq M$, so $T_1=\widetilde C_N(M)=\widetilde C_N(T^u)=\widetilde C_N(T)^u=T^u$. In particular, any definable subgroup of $N$ of dimension $1$ is commensurable either with $U$ or with a $U$-conjugate of $T$. Moreover, for any $g\in N$ there is $u\in U$ with $T^g=T^u$, so $N=U\,N_N(T)$.
This shows (\ref{lemma:item1}), once we know that $N_N(T)=T$ (which is proven below), since that implies $N=B$.

We say that a subset $X$ of $A$ is \emph{almost $H$-invariant} if there is some finite subset $H_0\subseteq H$ so that for any $h\in H$ we have $X^h\subseteq \bigcup_{h_0\in H_0}X^{h_0}$. Any finite almost $H$-invariant subset of $A$ is almost central in $\hat{B}$ and thus trivial. So, as $A$ is (finite and characteristic)-by-abelian, it is abelian. 

Now, let $R:={\rm End}_H(A)$ be the ring of endomorphisms of $A$ generated by $H$. Let $ r\in R \setminus \{0\}$. We show that $r$ is an automorphism of $A$: Since $H$ is finite-by-abelian, both ${\rm ker}(r)$ and ${\rm im}(r)$ are almost $H$-invariant definable subgroups of $A$ and hence either finite, and thus trivial, or finite index subgroups of the $1$-dimensional group $A$. Now ${\ker}(r)$ cannot have finite index in $A$ as otherwise ${\rm im}(r)$ is finite, whence trivial, so we would get $r=0$. So any $r\in R \setminus \{0\}$ is injective; since a definable injective map from a pseudofinite group to itself must be surjective, $R$ acts on $A$ by automorphisms. In particular it is invertible. By \cite[Proposition 3.6 or Corollary 3.10]{Wagner2020} the field of fractions $F$ of $R$ is an interpretable skew field; $A \cong F^+$ and $H \hookrightarrow F^\times$. Since ${\rm dim}(H)={\rm dim}(A)={\rm dim}(F)$, the image of $H$ has finite index in $F^\times$. As $F$ is pseudofinite and any finite skew field is commutative by Wedderburn’s Little Theorem, $F^\times$ is commutative and contains only finitely many elements in each finite order; clearly the same holds for $H$.

Finally, suppose $g\in N_N(T)$. Then $g$ induces an automorphism $\sigma$ of $F$ of finite order via conjugation: If $r\in R$ is given by $x\mapsto\prod_i x^{h_i}$, then $\sigma(r)$ is given by $x\mapsto\prod_i x^{h_i^g}$, and extends to $F$ by $\sigma(r^{-1}r')=\sigma(r)^{-1}\sigma(r')$. 
Let $F_0$ be the fixed field of $\sigma$. Then $[F:F_0]=o(\sigma)$ and $1={\rm dim}(F)=[F:F_0]\,{\rm dim}(F_0)\ge[F:F_0]$, as ${\rm dim}(F_0)\ge 1$. Hence $\sigma$ fixes $F$, so $g\in C_N(T/Z)\le\widetilde C_N(T)=T$. This shows (\ref{lemma:item2}), as well as $N_N(T)=T$ and $N=B$. So we have proven (\ref{lemma:item1}) and (\ref{lemma:item2}).

From now on, we think of $U$ and $T$ as the `unipotent' and the `semisimple' parts of $B$, respectively.

In order to show (\ref{lemma:item3}), assume that there is $g\in G$ so that $U^g \simeq T_C$. Then $\widetilde{N}_G(T_C)=\widetilde N_G(U^g)= B^g$. Therefore $T_C\leqslant B^g$; since $T_C$ is finite-by-abelian, $T_C\le \{x\in B^g:U^g\lesssim C_{B^g}(x)\}=U^g$. By Part (\ref{lemma:item2}), $T_C/Z_C$ contains only finitely many elements of any given order. But $T_C$ has finite index in $U^g$, and $U^g$ is finite-by-abelian. Therefore $U^g$ only contains finitely many elements of any given order.

Let $1\neq z\in Z_C$ be of order $\ell$ and put $\Omega_\ell=\{t\in U^g: o(t)=\ell\}$. Then $\Omega_\ell$ is finite and contains $z$. So $C_G(\Omega_\ell) \simeq N_G(\Omega_\ell) \ge B^g$. This implies $z\in Z_C\cap Z^g$, whence $ C\simeq C_G(z) \simeq B^g$ and $U^g\simeq T_C=\widetilde N_C(T_C)\simeq\widetilde N_{B^g}(U^g)=B^g$, a contradiction.

To show (\ref{lemma:item4}), consider $g\in G\setminus B$. Suppose that there is non-trivial $y\in U \cap U^g$, and put $C=C_G(y)$. Then $C$ almost contains $U$ and $U^g$ and must be of dimension $2$. Clearly $y\in Z_C=\widetilde{Z}(C)$, so neither $U$ nor $U^g$ can be commensurable with $T_C$ (the semisimple part of $C$) by Part (\ref{lemma:item3}). Hence both are commensurable with $U_C$, so $U\simeq U^g$ and $g\in \widetilde N_G(U)=B$, a contradiction. In particular $Z\cap Z^g \leqslant U\cap U^g=1$.

Next, note if char$(F)\not=2$ then $I(U\setminus Z)=\emptyset$, and if char$(F)=2$ then $I(T\setminus Z)=\emptyset$. In the latter case, since the $U$-conjugates of $T\setminus Z$ cover $B\setminus U$, we have $I(B)=I(U)$, so 
$I(B \cap B^g)\leqslant U\cap U^g=1$, and $B$ is strongly embedded.

To show (\ref{lemma:item5}), assume $Z\not=1$ and consider $g\in G\setminus B$.
As ${\rm dim}(B)={\rm dim}(B^g)=2$, ${\rm dim}(G)=3$ and $B\not\simeq B^g$, we have ${\dim}(B\cap B^g)=1$. Hence $B\cap B^g$ is commensurable with $U$ or a $U$-conjugate of $T$, and with $U^g$ or a $U^g$-conjugate of $T^g$. Clearly $U\simeq B\cap B^g\simeq U^g$ is impossible, as $g\notin\widetilde N_G(U)=B$. Part (\ref{lemma:item3}) excludes the mixed case, so $B\cap B^g\simeq T^u$ for some $u\in U$. Hence $$B\cap B^g \leqslant \widetilde{N}_B(B\cap B^g)=\widetilde{N}_B(T^u)=\widetilde{N}_B(T)^u=T^u.$$

Finally, consider $x\in B$ with ${\rm dim}(C_G(x))=2$. Note that $x\in \widetilde{Z}(C_G(x))$, which is finite by Lemma~\ref{lemma:nilp-by-finite}. So the order $o(x)$ is finite. If $x\in U$ then $ x\in U\cap U^c$ for any $c\in C_G(x)$, so $c\in B$ and $C_G(x)\le B$. Thus $x\in Z$.

Now suppose $x\in T\setminus Z$. Then $T\simeq C_B(x)$; by Parts (\ref{lemma:item1}) and (\ref{lemma:item3}) we have $C_G(x)\simeq U_x T_x$, where the unipotent part $U_x$ is not almost contained in $B$ and the semisimple part $T_x$ is commensurable to $T$. For any $u_x\in C_{U_x}(x)\setminus B$ we have $x\in B\cap B^{u_x}$ and $x\in (T\cap T^{u_x})\setminus (Z\cup Z^{u_x})$. But the $U$-conjugates of $T\setminus Z$ are disjoint and there are $u,u'\in U$ with $T^u\ge B\cap B^{u_x}\le T^{u'u_x}$. It follows that $u,u'\in N_U(T)=Z$, so $T\simeq T^{u_x}$ and $u_x\in\widetilde N_G(T)$. Thus $\widetilde{N}_{U_x}(T_x)=\widetilde N_{U_x}(T)$ is infinite, a contradiction. As the conjugates of $T$ cover $B\setminus U$, we are done.

This finishes the proof of the lemma.\end{proof}

From now on, we use the following notation.
\begin{notation*}Let $G$ be a group. For $g\in G$ we denote $C_g:=\widetilde C_G(C_G(g))$. Note that if $G$ is finite-dimensional with additive and fine dimension and if ${\rm dim}(C_G(g))=1$, then $C_G(g) \lesssim C_g$ by Theorem~\ref{th:infinite-abelian}.
\end{notation*}

\begin{lemma}\label{lemma:Ci}Let $G$ be an almost simple pseudofinite finite-dimensional group with additive and fine dimension. Assume that ${\rm dim} (G) \leqslant 3$. If $g\in G$ satisfies ${\rm dim}(C_G(g))=1$, then $C_g \simeq C_G(g)$ and $\widetilde N_G(C_g)=N_G(C_g)$. Moreover, if $i,j \in N_G(C_g)\setminus C_g$ are distinct involutions with ${\rm dim}(C_G(i))={\rm dim}(C_G(j))=1$, then $j \in iC_g$.\end{lemma}
\begin{proof}
Note that $C_G(g)$ is finite-by-abelian-by-finite, so $C_G(g)\lesssim C_g$. Suppose that $C_g$ had finite index in $G$. Since $C_g\lesssim \widetilde C_G(C_G(g)))$, by Fact \ref{fact:com:sym} we have $C_G(g)\lesssim\widetilde C_G(C_g)=\widetilde Z(G)$, contradicting triviality of $\widetilde Z(G)$. So $1\le {\rm dim}(C_g)< {\rm dim}(G) \leqslant 3$. Suppose ${\rm dim}(C_g)=2$. As $C_G(g)  \lesssim \widetilde{Z}(C_g)$ by Fact~\ref{fact:com:sym}, we get $\widetilde N_G(C_g)=G$ by Lemma \ref{lemma:nilp-by-finite}, contradicting almost normality. So ${\rm dim}(C_g)=1$ and $C_g \simeq C_G(g)$, whence $\widetilde N_G(C_g)=N_G(C_g)$. 

Now, let $i,j\in N_G(C_g) \setminus C_g$ be distinct involutions with $${\rm dim}(C_G(i))={\rm dim}(C_G(j))=1.$$ Let $Z \leqslant A$ be definable characteristic subgroups of $C_g$ so that $Z$ is finite and $A/Z$ is infinite and abelian (exist by Remark~\ref{remark:FItting}). By the above and our assumptions, ${\rm dim}(C_g)=1$, ${\rm dim}(C_G(i))=1$ and $C_G(i)\not\simeq C_G(g)$; hence $C_{A}(i)$ is finite, as is $C_{A/Z}(i)$. Then the subgroup of $A/Z$ inverted by $iZ$, say $B_i/Z$, contains $[r,A]Z/Z$ and thus is $1$-dimensional. Similar observations can be made for the group $B_j/Z$ of elements of $A/Z$ inverted by $jZ$; so $(B_i\cap B_j)/Z$ is infinite. Now $ijZ$ centralises $(B_i\cap B_j)/Z$; as $Z$ is finite there is $z\in Z$ such that $ijx_\ell ji=x_\ell z$ for infinitely many $x_\ell\in B_i\cap B_j$. Hence there is $1\neq x_0\in B_i\cap B_j$ so that $ij$ centralises $x_\ell x_0^{-1}$ and $C_{B_i\cap B_j}(ij)$ is infinite. Thus ${\rm dim}(C_{C_G(g)}(ij))=1$ and $ij\in C_g$, yielding $j \in iC_g$.
\end{proof}

\begin{lemma}\label{lemma:identification}Let $G$ be an almost simple pseudofinite finite-dimensional group with additive and fine dimension. Assume that ${\rm dim} (G) \leqslant 3$, and that, if ${\rm dim}(G)=3$, all definable subgroups of dimension $2$ are soluble-by-finite. Then ${\rm dim}(G)=3$ and $G$ has a definable subgroup of finite index isomorphic to ${\rm PSL}_2(F)$, where $F$ is a pseudofinite field.\end{lemma}

\begin{proof} Note first that $G$ is semisimple and has trivial almost centre by Lemma \ref{lemma:almost}; moreover it has dimension at least $2$. If ${\rm dim}(G)=2$ then all centralisers of non-trivial elements have dimension at most $1$ by Theorem~\ref{th:commensurable}. Note that we may assume that $m_2(G)>2$, that $G$ has no strongly embedded definable subgroup and that $G$ has involutions (Lemma~\ref{lemma:involutions}). Further, by Fact \ref{hempel-palacin} all involutions of $G$ have infinite centralisers.

\begin{claim}\label{claim:1} There is an involution $i\in G$ with ${\rm dim}(C_G(i))=1$, and for any such involution there is an involution $k\in N_G(C_i)\setminus C_i$. Moreover, if there is an involution $j\in G$ with ${\rm dim}(C_G(j))=2$, then we can choose $k\in C_G(i)\setminus C_i$.\end{claim}
\begin{proofclaim} Note first that if $i\in G$ is an involution with ${\rm dim}(C_G(i))=1$, then ${\rm dim}(\widetilde N_G(C_i))=1$. For otherwise, setting $H=\widetilde N_G(C_i)$ and $B=\widetilde N_G(H)$ we have ${\rm dim}(H)={\rm dim}(B)=2$. As ${\rm dim}(\widetilde N_B(C_i))=2$ we must have $C_i\simeq U$ in the notation of Lemma \ref{lemma:dim2subgroups}, so $i\in C_i=\widetilde C_G(C_i)=\widetilde C_G(U)=U$. Moreover ${\rm dim}(C_G(i))=1$ implies $i\notin \widetilde{Z}(B)$. But then $B$ is strongly embedded by Lemma \ref{lemma:dim2subgroups}(\ref{lemma:item4}), a contradiction.

Next, suppose that ${\rm dim}(C_G(i))=1$ for all involutions $i\in G$, and fix some involution $i$. Put $H=\widetilde N_G(C_G(i))$. Then $C_G(i)\lesssim C_i\le H$ are all commensurable, so $H=\widetilde N_G(H)$. Since $H$ is not strongly embedded, there is $g\in G\setminus H$ and an involution $k\in H\cap H^g$; if $k\in C_i\cap C_i^g$ then $C_G(k)$ almost contains $C_i$ and $C_i^g$; as $g\notin H=\widetilde N_G(C_i)$ the intersection $C_i\cap C_i^g$ is finite, so $C_G(k)$ has dimension $2$, a contradiction. If $k\notin C_i$ we are done, otherwise $k\notin C_i^g$ and $k^{g^{-1}} \notin C_i$ is as required.

Now suppose that $j\in G$ is an involution with ${\rm dim}(C_G(j))=2$. Note that this implies ${\rm dim}(G)=3$, as $\widetilde Z(G)=1$. We take $L=C_G(j)$ in Lemma \ref{lemma:dim2subgroups}, so $B=\widetilde N_G(L)$ and $Z=\widetilde Z(B)$. Then $j\in Z\le B$, so by Lemma \ref{lemma:dim2subgroups}(\ref{lemma:item4}) there is an involution $i\in T\setminus Z$, and we have ${\rm dim}(C_G(i))=1$ by Lemma \ref{lemma:dim2subgroups}(\ref{lemma:item5}). This shows the first part of the claim.

Let $i\in G$ be any involution with ${\rm dim}(C_G(i))=1$ and put $N=\widetilde N_G(C_i)$. Consider the conjugacy class $j^G$, which has dimension $1$. Recall that in any finite group two involutions are either conjugate or there is an involution commuting with both of them; by \L o\'s' Theorem this also holds in $G$. Clearly $i$ and $j^g$ are not conjugate for any $g\in G$, so for each $g\in G$ there is an involution $k_g\in C_G(i)\cap C_G(j^g)$.

We want to show that $k_g\notin C_i$ for some $g\in G$. So suppose otherwise. Now either ${\rm dim}(C_G(k_1))=2$, and we can take $B=\widetilde N_G(C_G(k_1))$. Then $i\in C_G(k_1)\le B$, so $C_B(i)$ is infinite, as is $C_i\cap B$. If $C_i\cap B\simeq U$, then  $U=\widetilde C_G(U)=\widetilde C_G(C_i)=C_i$ contains an involution $i$ with ${\rm dim}(C_G(i))=1$, so $i\notin Z$ and $B$ is strongly embedded, a contradiction. Hence $C_i\cap B=\widetilde C_G(C_i)\cap B=\widetilde C_G(T)\cap B=T$.

Otherwise ${\rm dim}(C_G(k_1))=1$, so $C_i=C_{k_1}$ and we can take $B=\widetilde N_G(C_G(j))$. The same argument (with $k_1$ and $j$ instead of $i$ and $k_1$) yields that $C_i\cap B=T$.
But $C_i$ is finite-by-abelian, $T$ has finite index in $C_i$ and $T$ contains only finitely many involutions (Lemma \ref{lemma:dim2subgroups}(\ref{lemma:item2})). Hence $C_i$ only contains finitely many involutions, and there is an involution $k\in C_i$ which commutes with infinitely many distinct conjugates~$j^g$.

Fix $j'\in C_G(k)\cap j^G$, put $B'=\widetilde N_G(C_G(j'))$ and $Z'=\widetilde Z(B')\ni j'$, and let $U'$ and $T'$ denote the unipotent and a semisimple subgroups of $B'$. Then $k\in C_G(j')\le B'$. Put $X=\{x\in G:k\in B'^x\}$. Then $X\cap j^G\supseteq C_G(k)\cap j^G$ is infinite; since $B'\cap j^G\subseteq Z'$ is finite, there is $x_0\in X\setminus B'$, and we may choose $T'\le B'$ so that $k\in B'\cap B'^{x_0}\le T'$ (recall that $\widetilde N_G(B')=B'$). As $Z'\cap B'\cap B'^{x_0}=1$ we have $k\notin Z'$, and $T'$ is the unique $U'$-conjugate of $B'$ containing $k$. Hence $T'$ does not depend on $x_0$. But for any $x\in X$ there is $u\in U'$ such that $B'\cap B'^x\le T'^u$, so clearly $T'^u=T'$. As $k$ is in some $U'^x$-conjugate of $T'^x$, there is $u'\in U'$ with $k\in T'^{u'x}\le B'^x$, so $T'\ge B'\cap B'^x\le T'^{u'x}$, whence $T'\simeq T'^{u'x}$ and $u'x\in\widetilde N_G(T')$. 
Moreover, if $x\in C_G(k)$ then we can take $u'=1$. Hence $C_G(k)\cap j^G\subseteq \widetilde N_G(T')$. It follows that $\widetilde N_G(T')$ cannot have dimension $2$, since this would imply $C_G(k)\cap j^G\subseteq\widetilde Z(\widetilde N_G(T'))$ which is finite. Therefore ${\rm dim}(\widetilde N_G(T'))=1$.

Put $B^\sharp=\bigcap_{g\in\widetilde N_G(T')}(T')^g$, a normal subgroup of $\widetilde N_G(T')$ of finite index contained in all $B'^x$ for $x\in X$. Now $k\notin Z'$, so $C_G(k)\simeq T'$ and $C_G(k)\lesssim B^\sharp$. Hence there is fixed $j'^{x'}\in C_G(k)$ and infinitely many distinct $j'^{x_\ell} \in C_G(k)$ so that $j'^{x'}=j'^{x_\ell}t_\ell$ for some $t_\ell\in B^\sharp$. Now $j'^{x'}\in B'^{x'}$, whence $k\in C_G(j'^{x'})\le B'^{x'}$ and $x'\in X$. Also $t_\ell\in B^\sharp\le B'^{x'}$, whence $j'^{x_\ell}\in B'^{x'}$. So $j'^{x_\ell}\in \widetilde{Z}(B'^{x'})$ by Lemma~\ref{lemma:dim2subgroups}(\ref{lemma:item5}), contradicting finiteness of $\widetilde{Z}(B'^{x'})$. This shows that there is an involution $k\in C_G(i)\setminus C_i$ commuting with $i$.\end{proofclaim}
We now fix an involution $i\in G$ with ${\rm dim}(C_G(i))=1$ and put $N=N_G(C_i)$.
\begin{claim}\label{claim:2}An involution $k\in N\setminus C_i$ centralises a unique involution $\ell\in C_i$, and ${\rm dim}(C_G(k))={\rm dim}(C_G(\ell))=1$.\end{claim}
\begin{proofclaim} In a finite group $N$, if $C$ is normal in $N$ and both $C$ and $N\setminus C$ contain involutions, then for any involution $k\in N\setminus C$ a Sylow 2-subgroup $P$ containing $k$ must intersect $C$ non-trivially, as all Sylow 2-subgroups are conjugate. Then $Z(P)\cap C$ is non-trivial by nilpotency of $P$, and $C$ contains an involution $j$ commuting with $k$. By \L o\'s' Theorem, the same holds for $N> C_i$, and there is an involution $\ell\in C_i$ commuting with $k$.

Next we show ${\rm dim}(C_G(k))={\rm dim}(C_G(\ell))=1$. This is clear if ${\rm dim}(G)=2$, so suppose ${\rm dim}(G)=3$.

Assume first that ${\rm dim}(C_G(k))=2$. Then ${\rm dim}(\widetilde N_G(C_G(k)))=2$ by almost simplicity of $G$. If ${\rm dim}(C_G(\ell))=2$, then, as $\ell\in C_G(k)\le \widetilde N_G(C_G(k))$, we get $\ell\in\widetilde Z(N_G(C_G(k)))$.
Hence $C_G(i)\lesssim C_G(\ell)\simeq C_G(k)$, contradicting $k\notin C_i$. If ${\rm dim}(C_G(\ell))=1$, then $C_G(i)\simeq C_G(\ell)\lesssim C_G(k)$, again a contradiction. Thus ${\rm dim}(C_G(k))=1$.

Next, suppose ${\rm dim}(C_G(\ell))=2$ and put $B''=\widetilde N_G(C_G(\ell))$. Then $i,k\in C_\ell \le B''$, so $C_G(i)$ and $C_G(k)$ are commensurable to two different $B''$-conjugates of the semisimple part $T''$ of $B''$. As $k\notin\widetilde Z(B'')$, this implies $k\notin N$, since $\widetilde{N}_{B''}(T'')=T''$. Hence ${\rm dim}(C_G(\ell)) =1$. 

Finally, we show uniqueness of $\ell$. So assume that $k$ fixes another involution $\ell'\in C_i$. Then ${\rm dim}(C_G(\ell'))=1$; by Lemma \ref{lemma:Ci} we have $\ell'\in\ell C_k$, so $\ell\ell'\in C_k$ and $C_G(k) \lesssim  C_G(\ell \ell')$. Moreover $C_G(i) \lesssim  C_G(\ell \ell')$. So $C_G(\ell \ell')$ is $2$-dimensional. If ${\rm dim}(G)=2$ this is a contradiction; if ${\rm dim}(G)=3$ we put $B'''=\widetilde N_G(C_G(\ell\ell'))$. Then $C_G(k)$ and $C_G(\ell) \simeq C_G(i)$ are commensurable to two different $B'''$-conjugates of the semisimple part $T'''$ of $B'''$, which again implies $k\notin N$. Thus $k$ centralises a unique involution in $ C_i$.\end{proofclaim}

\begin{claim}Any two involutions $x,y$ in $C_i$ with ${\rm dim}(C_G(x))={\rm dim}(C_G(y))=1$ commute.\end{claim}
\begin{proofclaim}
Let $k_x,k_y \in N\setminus C_i$ be involutions commuting with $x,y$ respectively. Now $C_i$ is normal in $N$, so $y$ and $k_x$ are not $N$-conjugate, and there is an involution $z\in N$ commuting with both. If $z\in C_i$ then, as it is centralised by $k_x$, we have $z=x$, so $x$ and $y$ commute. If $z\in N\setminus C_i$ then either $z=k_x$, or $zk_x$ is an involution in $ C_i$ centralised by $k_x$, whence $zk_x=x$. In either case $z$ commutes with $x$ and with $y$, whence $x=y$. So $x$ and $y$ commute.\end{proofclaim}

\begin{claim}\label{claim:4}If $x,y,z\in C_i$ are three distinct involutions with $${\rm dim}(C_G(x))={\rm dim}(C_G(y))={\rm dim}(C_G(z))=1,$$ then $xyz=1$.\end{claim}
\begin{proofclaim} Let $k\in N\setminus C_i$ be an involution commuting with $x$. Then $y$ and $y^k$ are distinct commuting involutions in $ C_i$. Therefore $yy^k=(yk)^2$ is an involution in $ C_i$ fixed by $k$, whence $(yk)^2=x$. Similarly, $(zk)^2=x$, and $k^y=k^z=xk$. So $yz$ is an involution in $ C_i$ commuting with $k$, whence $yz=x$ and $xyz=1$.\end{proofclaim}

\begin{claim}\label{claim:5}${\rm dim}(C_G(j))=1$ for any involution $j\in G$.\end{claim}
\begin{proofclaim} Suppose, as in the proof of Claim \ref{claim:1}, that $j$ is an involution with ${\rm dim}(C_G(j))=2$. Note that we may assume that $i,j\in B$ commute (where $B$ is again as in Claim \ref{claim:1}, whence as in Lemma~\ref{lemma:dim2subgroups}): If not, then there is a third involution $z\in C_G(j)\cap C_G(i)$. If $z\in\widetilde Z(B)$ we can replace $j$ by $z$. Otherwise $z\in T\setminus Z$, and we can replace $i$ by $z$. Note that $ij\in T\setminus\widetilde Z(B)$, so ${\rm dim}(C_G(ij))=1$.

Consider $k\in C_G(i)\setminus C_i$  (exists by Claim~\ref{claim:1}). Then $B\cap B^k\simeq T$, and $\ell=j^k\in( C_i\cap C_G(i))\setminus B$. Then $i$, $ij$ and $i\ell$ are distinct involutions in $ C_i$ with ${\rm dim}(C_G(i))={\rm dim}(C_G(ij))={\rm dim}(C_G(i\ell))=1$. Then, by the above, $iiji\ell=1$ whence $j=i\ell$, but $C_G(j)$ is $2$-dimensional, a contradiction.\end{proofclaim}

Now, since $m_2(G)>2$ there are three pairwise commuting distinct involutions $x,y,z$ with $xyz\not=1$. By Claim \ref{claim:4} and~\ref{claim:5} we cannot have $C_x=C_y=C_z$.
If $C_x=C_y\not=C_z$, then $z\in N_G(C_x)\setminus C_x$ centralises two distinct involutions $x,y\in C_x$, contradicting Claim \ref{claim:2}. If $C_x\not=C_y\not= C_z\not= C_x$, then $y,z\in N_G(C_x)$, so $y\in zC_x$ by Lemma \ref{lemma:Ci}, and $yz\in C_x$ is an involution centralised by $y$, whence $yz=x$ again by Claim \ref{claim:2}, contradicting our hypothesis.

This finishes the proof.
\end{proof}

\subsection{Proof of Theorem~\ref{th:rank2}}
\begin{customthm}{1.2}Let $G$ be a finite-dimensional pseudofinite group with additive and fine dimension. If ${\rm dim}(G) = 2$ then $G$ is soluble-by-finite. Moreover, if $G$ is not finite-by-abelian-by-finite, then there is a definable normal subgroup $N$ of $G$ with ${\rm dim}(N)=1$.
\end{customthm}

\begin{proof} Suppose first that $G$ has a definable normal subgroup $N$ of dimension $1$. Then $N$ is finite-by-abelian-by-finite. In fact $\widetilde Z(N)'$ is finite normal in $G$, and $C_N(\widetilde Z(N)')$ is a normal $2$-nilpotent subgroup of $G$ of dimension $1$. The quotient $\overline G=G/C_N(\widetilde Z(N)')$ is also of dimension $1$, whence finite-by-abelian-by-finite. Put $\overline S=C_{\overline G}(\widetilde Z(\overline G)')$. Then $\overline S$ is $2$-nilpotent of finite index in $\overline G$, and its preimage $S$ is soluble of finite index in $G$. Thus $G$ is soluble-by-finite and has a definable normal subgroup of dimension $1$.

Now suppose that there is no definable normal subgroup of dimension $1$. 
Every finite normal subgroup of $G$ is contained in $\widetilde Z(G)$, which is finite-by-abelian.

If ${\rm dim}(\widetilde Z(G))=2$ we are done and $G$ is finite-by-abelian-by-finite.
So suppose that $\widetilde Z(G)$ is finite. Then $\overline G=G/\widetilde Z(G)$ has no non-trivial finite normal subgroup. Replacing $G$ by $\overline G$, we may thus assume that $G$ is almost simple. So we may apply Lemma~\ref{lemma:identification}, which implies that $G$ has dimension $3$, a contradiction.\end{proof}

\subsection{Proof of Theorem~\ref{th:main}}

\begin{customthm}{1.1}[Without CFSG]
Let $G$ be a pseudofinite finite-dimensional group with additive and fine dimension. If ${\rm dim} (G) = 3$, then either $G$ is soluble-by-finite, or $\widetilde Z(G)$ is finite and $G/\widetilde Z(G)$ has a definable subgroup of finite index isomorphic to ${\rm PSL}_2(F)$ where $F$ is a  pseudofinite field.\end{customthm}

\begin{proof} Suppose first that $G$ has a definable normal subgroup $N$ with $1\le{\rm dim}(N)\le 2$. Then $1\le{\rm dim}(G/N)\le 2$, so by Theorem \ref{th:rank2} both $N$ and $G/N$ are soluble-by-finite. Put $S={\rm Rad}(N)$, which is characteristic in $N$ and hence normal in $G$. Then $C_G(N/S)$ is a definable normal soluble subgroup of finite index in $G$.

So suppose that $G$ has no definable normal subgroup of dimension $1$ or $2$. If ${\rm dim}(\widetilde Z(G))=3$, then $G$ is finite-by-abelian-by-finite. Otherwise $\widetilde Z(G)$ is finite and contains all normal finite subgroups. Replacing $G$ by $G/\widetilde Z(G)$, we may assume that $G$ is almost simple. By Theorem \ref{th:rank2} all definable subgroups of dimension $2$ are soluble-by-finite. So we finish by Lemma \ref{lemma:identification}.\end{proof}

We finish the paper by stating the immediate corollary of Theorem~\ref{th:main}.

\begin{corollary}[Without CFSG]\label{corol:main}Let $G$ be a pseudofinite non-abelian definably simple finite-dimensional group with additive and fine dimension. If ${\rm dim} (G) = 3$, then $G\cong {\rm PSL}_2(F)$ for some pseudofinite field $F$. \end{corollary}

\addcontentsline{toc}{section}{References}    
\bibliographystyle{plain}

\end{document}